\documentclass[10pt,reqno]{amsart}
\usepackage{graphicx,amssymb,mathrsfs,amsmath,color,amsthm,fancyhdr}
\newtheorem{lem}{Lemma}
\newtheorem{rem}{Remark}

\newtheorem{theo}{Theorem}
\title[The Bergman kernel of a certain Hartogs domain]{\ The Bergman kernel of a certain Hartogs domain and the polylogarithm function}
\author{Atsushi Yamamori}
\address{Graduate School of Mathematics, Nagoya University, Furo-Cho, Chikusa-Ku, Nagoya 464-8602, Japan}
\begin{document}
\email{d08006u@math.nagoya-u.ac.jp}
\begin{abstract}
We consider a certain Hartogs domain which is related to the Fock-Bargmann space.
We give an explicit formula for the Bergman kernel of the domain in terms 
of the polylogarithm functions.
Moreover we solve the Lu Qi-Keng problem of the domain in some cases.
\end{abstract}
\maketitle
\section{Introduction}
In this paper we consider a Hartogs domain in $\mathbb{C}^{n+m}$ defined by the inequality
$||\zeta||^2<e^{-\mu||z||^2  }, $
where $(z,\zeta)\in\Bbb{C}^n\times\Bbb{C}^m$ and $\mu>0$.
Our aim is to show that the Bergman kernel of this domain can be written explicitly
in terms of the polylogarithm functions.\par
The polylogarithm function appears in many different areas of mathematics.
For example it appears in analysis of the Riemann zeta function, algebraic
geometry and
mathematical physics (cf. \cite{d2010}, \cite{Hirzebruch2008}). 
The polylogarithm function is a rational function under a certain condition.
Our formula is expressed in terms of these rational
cases of the polylogarithm functions and their derivatives.\par
It is usually hard to obtain an explicit formula of the Bergman kernel
of a complex domain.
Only few domains with an explicit Bergman kernel are known until now.
In this situation, it is fundamental and important to find a domain with
explicit Bergman kernel.\par
There are two basic approaches for obtaining explicit Bergman kernels.
One is to construct a complete orthonormal basis of the Bergman space explicitly. 
For the unit disk, one can find a complete orthonormal basis
$\{\pi^{-1/2}(n+1)^{1/2}z^n \}_{n=1}^\infty$.
However this approach faces difficulty in general. If the domain does not have symmetry,
a computation of integral on the domain is unexecutable or extremely difficult.\par
If the automorphism group of a domain contains enough information (e.g.
transitivity) to obtain explicit formula,
then we can use it. 
For example, the Lie group $SU(n,m)$
acts transitively on the classical domain $D=\{z\in M_{n,m}(\mathbb{C})
;I-\overline{z}{^t}z>0 \}$ of type I by linear fractional transformation.
It is known that the Bergman kernel $K$ of a classical domain has the property
that
$K(z,0)$ is a non-zero constant.
These facts and a transformation rule of the Bergman kernel (cf.
\cite{Boas2000}) imply that the computation of the Bergman kernel is reduced to the computation
of the Jacobian of the linear fractional map.
L. K. Hua \cite{Hua1964} computed the Bergman kernels for the classical domains in
this way.
As above, this approach faces difficulty in general.\par
The approach in this article is different from the above two.
Our method is based on Ligocka's theorem \cite{Ligocka1989} which relates the Bergman kernel of a Hartogs domain to
weighted Bergman kernels of the base domain.
Thanks to this theorem one can find that  our domain and the Fock-Bargmann space
are closely related.
We will see that Ligocka's theorem and an explicit formula of the Fock-Bargmann kernel
lead to an explicit formula of the Bergman kernel of our domain.\par
As an application of our formula, we solve the Lu Qi-Keng problem for our domain in some cases.
The Lu Qi-Keng problem asks whether the Bergman kernel has zeros or not.
This problem was investigated for various domains by many authors in this decade.
Yin \cite{WeiPing1999} obtained explicit form of the Bergman kernel of the Cartan-Hartogs domain.
The Lu Qi-Keng problem for the Cartan-Hartogs domain was studied by several authors (cf. \cite{Demmand2009}).
In \cite{Wang2009}, the authors obtained an explicit formula of the Bergman kernel of some Hartogs domains and solved the Lu Qi-Keng problem
for the domains in some cases. Recently Lu Qi-Keng himself studied the location
of the zeros of Bergman kernel in \cite{Lu2009}.
Further information about the Lu Qi-Keng problem can be found in
\cite{Boas2000},\cite{Boas1999},\cite{Jarnicki2005} and \cite{WeiPing2008}.

\section{Preliminaries}
Let $\Omega$ be a domain in $\Bbb{C}^n$, $L^2_a(\Omega)$ the Hilbert space of square integrable holomorphic
functions on $\Omega$
with the inner product:
$$ \langle f,g \rangle=\int _{\Omega} f(z)\overline{g(z)} dz,\mbox{\quad
for all $f,g\in L^2_a(\Omega)$. }$$
The Bergman kernel $K(z,w)=\overline{K_z(w)}$ is the reproducing kernel
for $L_a^2(\Omega)$,
i.e. if $f\in L_a^2(\Omega)$ then
$$f(z)= \langle f,K_z \rangle=\int _{\Omega} f(w)K(z,w) dw, \mbox{\quad for
all $z\in\Omega$. }$$
Let $\{\phi_k\}$ be a complete orthonormal basis of $L_a^2(\Omega)$. Then the Bergman kernel
can be also defined by
$$ K(z,w)=\sum_k\phi_k(z)\overline{\phi_k(w)}. $$
Let $p$ be a positive continuous function on $\Omega$ and $L^2_a(\Omega,p)$
the Hilbert space of square integrable holomorphic
functions with respect to the weight function $p$ on $\Omega$ with the inner
product 
$$ \langle f,g \rangle=\int _{\Omega} f(z)\overline{g(z)} p(z) dz,\mbox{\quad
for all $f,g\in L^2_a(\Omega)$. }$$
The weighted Bergman kernel $K_{\Omega, p}$ of $\Omega$ with respect to the
weight $p$ is the
reproducing kernel of $L^2_a(\Omega,p)$.\par
We define the Hartogs domain $\Omega_{m,p}$ by
\begin{align*}
 \Omega_{m,p}:=  \{ (z,\zeta)\in\Omega\times\Bbb{C}^m; ||\zeta||^2 <
p(z)\}.
\end{align*}
E. Ligocka \cite[Proposition 0]{Ligocka1989} showed that
the Bergman kernel of $\Omega_{m,p}$ is expressed as infinite sum of
weighted Bergman kernels of the base domain $\Omega$.
\begin{theo}
Let $K_m$ be the Bergman kernel of $\Omega_ {m ,p }$ and $K_{\Omega,p^k}$
the
weighted Bergman kernel of
$\Omega$ with respect to the weight function $p^k$. Then
\begin{align*}
 K_m((z,\zeta),(z',\zeta') ) =\dfrac{m!}{\pi^m}\sum_{k=0}^\infty
\dfrac{ (m+1)_k}{k!}
K_{\Omega,p^{k+m}}(z,z')\langle\zeta,\zeta'\rangle^k.
\end{align*}
Here $(a)_k$ denotes the Pochhammer symbol
$(a)_k =a(a+1)\cdots (a+k-1).  $
\end{theo}
M. Engli\v{s} and G. Zhang generalized this theorem for wider class of domains
in \cite{Englivs2006}.
Since theorem of this type was first proved by F. Forelli and W. Rudin \cite{FR} for
$\Omega$ the unit disk and $p(z)=1-|z|^2$,
some authors call it the Forelli-Rudin construction.
\par
We introduce the polylogarithm function which is necessary to state our main theorem.
The polylogarithm function is defined by
\begin{align}
 Li_s(z)=\sum_{k=1}^\infty k^{-s}z^k ,
 \end{align}
which converges for $|z|<1$ and any $s\in\Bbb C$.
If $s$ is a negative integer, say $s=-n$, then the polylogarithm function has
the following closed form:
\begin{align*}
  Li_{-n}(z)&=
\dfrac{ z }{(1-z)^{n+1}}
 \sum_{j=0}^{n-1}  A(n,j+1)z^{j}
\end{align*}
where $A(n,m)$ is the Eulerian number \cite[eq.(2.17)]{d2010}\begin{align*}
 A(n,m)=\sum_{\ell=0}^{m} (-1)^\ell \binom{n+1}{\ell} (m-\ell)^n.
\end{align*}
The first few are
$$\begin{array}{ll}
 Li_{-1}(z)=\dfrac{z}{(1-z)^2}, &  Li_{-2}(z)=\dfrac{z^2 + z}{(1-z)^3}, \\
 Li_{-3}(z)=\dfrac{z^3 + 4z^2+ z}{(1-z)^4}, &  Li_{-4}(z)=\dfrac{z^4 +11z^3 +11z^2 +z}{(1-z)^5}.
\end{array}$$ 
The polynomial $A_n(z) = \sum_{j=0}^{n-1}  A(n,j+1)z^{j}$ is called the Eulerian
polynomial.
More information about the polylogarithm function and the Eulerian polynomial can be found in \cite{Carlitz},\cite{d2010} and \cite{Hirzebruch2008}.

\section{The Bergman kernel of $D_{n,m}$}

Let $\mu>0$. Define $D_{n,m}$ by
$$D_{n,m}:=  \{ (z,\zeta)\in\Bbb{C}^n\times\Bbb{C}^m; ||\zeta||^2 <
e^{-\mu ||z||^2}\}.$$
This section is devoted to the study of the Bergman 
kernel of $D_{n,m}$.
We shall begin with the Fock-Bargmann space and its reproducing
kernel.\par
The Fock-Bargmann space $L_a^2(\Bbb{C}^n,e^{-\mu||z||^2})$ is
 the Hilbert space of square integrable entire functions on $\Bbb{C}^n$ with the
inner product
$$\langle f,g\rangle =\int_{\Bbb{C}^n }f(z)\overline{g(z)} e^{-\mu||z||^2} dz .$$
The reproducing kernel $K_{n,\mu}$ of $L_a^2(\Bbb{C}^n,e^{-\mu|| z||^2})$ is expressed explicitly as
\begin{align}
 K_{n,\mu}(z,w)=\dfrac{ \mu^n e^{\mu \langle z,w\rangle}}{\pi^n}.
\end{align}
The kernel function $K_{n,\mu}$ is called the Fock-Bargmann kernel (see
\cite{bargmann}).
We are now ready to state our main result.
\begin{theo}
The Bergman kernel of $D_{n,m}$ is given by
\begin{align}
 K_{D_{n,m}}((z,\zeta),(z',\zeta') ) &=
\dfrac{\mu^n}{\pi^{n+m}}
e^{m\mu \langle z,z'\rangle }\dfrac{d^m}{d t^m} Li_{-n}(t)\lvert_{t=e^{\mu\langle 
z,z'\rangle}\langle \zeta ,\zeta'  \rangle  }\\
&=\dfrac{\mu^n}{\pi^{n+m}} \dfrac{d^{m-1}}{d t^{m-1}} 
\dfrac{Li_{-(n+1)} (e^{\mu\langle z,z'\rangle }t)  }{t}
\lvert_{t=\langle \zeta ,\zeta'  \rangle  }.
\end{align}
\end{theo}
\begin{proof}
By Ligocka's theorem and the formula $(2)$, we have
\begin{align*}
 K_{D_{n,m}}((z,\zeta),(z',\zeta') )
& =\dfrac{ m! }{ \pi^{m} } \sum_{k=0}^\infty
 \dfrac{(m+1)_k}{k!} \dfrac{(k+m)^n\mu^n}{\pi^n}  e^{\mu(k+m)\langle z,z' 
\rangle}
\langle \zeta,\zeta'\rangle^k\\  
&=\dfrac{ m!\mu^n }{ \pi^{n+m} } e^{\mu m\langle z,z' \rangle}
\sum_{k=0}^\infty \dfrac{(m+1)_k}{k!} (k+m)^n
  e^{\mu k\langle z,z' \rangle}\langle \zeta,\zeta'\rangle^k.
\end{align*}
Using a simple identity $(m+1)_k/k!=(k+1)_m/m!$,  we get
\begin{align*}
 K_{D_{n,m}}((z,\zeta),(z',\zeta') )
& =\dfrac{\mu^n}{ \pi^{n+m} } e^{\mu m\langle z,z' \rangle} 
\sum_{k=0}^\infty
 (k+1)_m (k+m)^n  e^{\mu k\langle z,z' \rangle}
\langle \zeta,\zeta'\rangle^k.
\end{align*}
Here we remark that
\begin{align}
|e^{\mu \langle z,z'\rangle} \langle\zeta,\zeta'\rangle |<1
\end{align}
for all $(z,\zeta),(z',\zeta') \in D_{n,m}$.
Indeed, 
from the definition of $D_{n,m}$ and the Cauchy-Schwartz inequality, we see that
$ |\left\langle \zeta,\zeta'  \right\rangle  |^2  \leq ||\zeta||^2||
\zeta'||^2<
e^{-\mu (||z||^2 +||z'||^2)},  $
for any $(z,\zeta),(z',\zeta') \in D_{n,m}$.
Combining this and a simple inequality $||z||^2+||z'||^2 \geq
2\mbox{Re}\left\langle z,z'  \right\rangle $,
we have $|\left\langle \zeta,\zeta' \right\rangle  |^2 < |e^{-\mu\left\langle
z,z'  \right\rangle  }|^2$. Hence $|e^{\mu \langle z,z'\rangle}
\langle\zeta,\zeta'\rangle |<1$. \par
Let us evaluate the series
\begin{align}
H_{m,n}((z,\zeta),(z',\zeta'))= \sum_{k=0}^\infty
 (k+1)_m (k+m)^n  e^{\mu k\langle z,z' \rangle}
\langle \zeta,\zeta'\rangle^k.
\end{align}
It is easy to see from (1) that the $m$-th derivative of the polylogarithm function
has the following series representation:
\begin{align}
\dfrac{d^m Li_s (z) }{d z^m} &=\sum_{k=m}^\infty (k-m+1)_m k^{-s} z^{k-m}\\
&= \sum_{k=0}^\infty (k+1)_m
(k+m)^{-s}z^k, 
\end{align}
for $|z|<1$.
Comparing (6) and (8), we obtain
$$H_{m,n}((z,\zeta),(z',\zeta'))=  \dfrac{d^m}{d t^m} 
Li_{-n}(t)
\lvert_{t=e^{\mu\langle z,z'\rangle}\langle \zeta ,\zeta'  \rangle  }.$$
This proves the formula $(3)$. The formula $(4)$ follows from $(3)$ 
and a well-known property of the polylogarithm function \cite[eq. 2.1]{d2010}:
$$ \dfrac{d}{d t} Li_s (t)=\dfrac{Li_{s-1}(t)}{t} .$$
We have just completed the proof of Theorem 2.
\end{proof}
\begin{rem}
There is a following closed form of the $m$-th derivative of the polylogarithm
function: 
\begin{align}
\dfrac{d^m Li_{-n} (t) }{d t^m} = \frac{m!\sum_{j=0}^{n} (-1)^{n+j}  (m+1)_j
S(1+n,1+j)(1-t)^{n-j}   }{(1-t)^{n+m+1}}  ,
\end{align}
where $S(\cdot,\cdot)$ denotes the Stirling number of the second kind (see \cite{d2010}).\par
\end{rem}

\section{ An application}

As an application of Theorem 2 we solve the Lu Qi-Keng problem for $D_{n,m}$ in some cases.
The Lu Qi-Keng problem asks whether the Bergman kernel has zeros or not.
He posed this problem in connection with
the global well-definedness of the representative coordinates (see \cite{Boas2000}).
Lu Qi-Keng's recent result \cite{Lu2009} implies that the zero of the Bergman
kernel has a geometric interpretation.\par
We begin with the following lemma which
together with the inequality (5) tells us that
the image of the
map $D_{n,m}\times D_{n,m}\ni ((z,\zeta),(z',\zeta')) \mapsto e^{\mu \langle
z,z'  \rangle} \langle \zeta,\zeta' \rangle   \in \mathbb{C}  $ is the unit disk.

\begin{lem}
 For any $\alpha\in\mathbb{C}$ such that $|\alpha|<1$, there exist $(z,\zeta) ,(z',\zeta')\in D_{n,m}$
such that
 $\alpha=e^{\mu\langle z,z' \rangle}\langle \zeta,\zeta'\rangle$.
\end{lem}
\begin{proof}
Let $\alpha=re^{i\theta}, r<1.$
For any fixed $z\in\mathbb{C}^n$, we can choose
$\zeta\in\mathbb{C}^m$ such that
$|| \zeta||^2=re^{-\mu ||z||^2} $.
Then $(z,\zeta),(z, e^{-i\theta}\zeta) \in
D_{n,m}$ and
 $e^{\mu\left\langle z,z  \right\rangle  }\left\langle \zeta,e^{-i \theta}\zeta
\right\rangle =r e^{i\theta}$.
\end{proof}
We next discuss the location of zeros of the polylogarithm function and its
derivative.
\begin{lem}
The function $Li_{-n}(z)/z$ has a zero $z_0$ such that $|z_0|<1$
for all $n \geq 3$.
\end{lem}
     \begin{proof}
      It is well-known that the Eulerian polynomial has only negative real, simple roots (see \cite[p. 292, Exercise 3]{comtet}).
Since $n\geq3$, there exists a root $\alpha$ such that $|\alpha|\neq1$.
If $|\alpha|<1$, then $\alpha$ is a desired zero. Now we assume that
$|\alpha|>1$.
Then the following formula \cite[eq.(2.2)]{d2010}
$$ Li_{-n}\left( \dfrac{1}{z} \right) =(-1)^{n+1}Li_{-n}(z)  \quad(n\in
\mathbb{N}),$$
implies that $\alpha^{-1}$ is a desired zero.
     \end{proof}
Further information of the location of zeros of $Li_{-n}(z)/z$ is found in \cite{Peyerimhoff1966}.
This short proof was obtained by private communication with Prof. Ochiai and Dr. Shiomi (compare with the proof in \cite{Peyerimhoff1966}).\par
The following is immediate from a straightforward computation.
\begin{lem}
For any $m\in\Bbb{N}$, the $(m-1)$-th derivative of $Li_{-2}(t)/t$ is expressed
as
 $$\dfrac{d^{m-1}}{dt^{m-1}} \dfrac{Li_{-2}(t)}{t}=\dfrac{(m+1)!(t+m)}{(1-t)^{m+2}}. $$ 
\end{lem}
From this lemma, we see that $t=-m$ is the zero of
$\frac{d^{m-1}}{dt^{m-1}} \frac{Li_{-2}(t)}{t}$.

Summarizing, we get:
\begin{theo}
The Bergman kernel $K_{D_{n,m}}$ is zero-free if
$n=1$ and $m\geq1$.
If $m=1$ and $n\geq 2$ then $K_{D_{n,m}}$ has a zero.
\begin{rem}
For our domain $D_{n,m}$, the solution of the Lu Qi-Keng problem depends only on
the value $(m,n)$.
In general, the solutions of the Lu Qi-Keng problem for the Hartogs domains
$ \{ (z,\zeta)\in\Omega\times\Bbb{C}^m; ||\zeta||^2 <p(z)^\mu\}$
depend not only on  $(m,n)$ but also on $\mu$ (cf. \cite{Demmand2009}).
\end{rem}

\end{theo}
\section*{Acknowledgements}
The author would like to express sincerest gratitude to Professors Hideyuki Ishi and
Hiroyuki Ochiai and Dr. Daisuke Shiomi for their helpful advices and
discussions.
The author also acknowledges the encouragement and helpful comments on this paper of Professor Takeo Ohsawa.

\bibliographystyle{plain}

\end{document}